\documentclass{amsart}

\usepackage{amsmath,amsfonts,amsthm,amssymb,verbatim,enumerate,graphicx,mathtools}
\usepackage{color}
\usepackage[flushmargin]{footmisc}
\newtheorem{theorem}{Theorem}[section]
\newtheorem{lemma}[theorem]{Lemma}
\newtheorem{corollary}[theorem]{Corollary}
\newtheorem{definition}[theorem]{Definition}

\newtheorem{remark}{Remark}
\numberwithin{equation}{section}
\def \isnatural {\in\mathbb{N}}

\newcommand{\tef}{transcendental entire function}
\newcommand\qfor{\quad\text{for }}

\newcommand \C{\mathbb{C}}

\newcommand \R{\mathbb{R}}

\newcommand{\spw}{spider's web}
\newcommand{\spws}{spiders' webs}
\DeclarePairedDelimiter{\maxnorm}{|\!|}{|\!|_\infty}
\makeatletter
\def\blfootnote{\xdef\@thefnmark{}\@footnotetext}
\makeatother
\begin{document}
%
%
%
%
\title[Hollow quasi-Fatou components of quasiregular maps]{Hollow quasi-Fatou components of quasiregular maps}
\author{Daniel A. Nicks, \, David J. Sixsmith}
\address{School of Mathematical Sciences \\ University of Nottingham \\ Nottingham
NG7 2RD \\ UK}
\email{Dan.Nicks@nottingham.ac.uk}
\address{School of Mathematical Sciences \\ University of Nottingham \\ Nottingham
NG7 2RD \\ UK}
\email{David.Sixsmith@nottingham.ac.uk}
%
%
%
%
\begin{abstract}
We define a quasi-Fatou component of a quasiregular map as a connected component of the complement of the Julia set. A domain in $\R^d$ is called \emph{hollow} if it has a bounded complementary component. We show that for each $d \geq 2$ there exists a quasiregular map of transcendental type $f: \R^d \to \R^d$ with a quasi-Fatou component which is hollow. 

Suppose that $U$ is a hollow quasi-Fatou component of a quasiregular map of transcendental type. We show that if $U$ is bounded, then $U$ has many properties in common with a multiply connected Fatou component of a {\tef}. On the other hand, we show that if $U$ is not bounded, then it is completely invariant and has no unbounded boundary components. We show that this situation occurs if $J(f)$ has an isolated point, or if $J(f)$ is not equal to the boundary of the fast escaping set. Finally, we deduce that if $J(f)$ has a bounded component, then all components of $J(f)$ are bounded.
\end{abstract}
\maketitle
%
%
%
%
\blfootnote{Both authors were supported by Engineering and Physical Sciences Research Council grant EP/L019841/1.}
\section{Introduction}
In the study of complex dynamics, the first example of a {\tef} with a multiply connected Fatou component was given by Baker \cite{MR0153842} over half a century ago. We refer to the survey \cite{MR1216719} for definitions and further background on complex dynamics. Since Baker's paper many authors have studied multiply connected Fatou components; see, for example, the papers \cite{MR0419759, MR759304, MR3149847, MR3145128, MR2178719}.

In this paper we extend this study to more than two (real) dimensions, for the first time. Suppose that $d \geq 2$, and that $f : \R^d \to \R^d$ is a \emph{quasiregular map of transcendental type}. We refer to Section~\ref{Sdefs} for a definition of a quasiregular map, which is said to be of transcendental type if it has an essential singularity at infinity. In this setting we need a different definition of the Julia set to that used in complex dynamics. Following \cite{MR3009101} and \cite{MR3265283}, we define the \emph{Julia set $J(f)$} to be the set of all $x \in \R^d$ such that
\begin{equation}
\label{Juliadef}
\operatorname{cap} \left(\R^d\backslash \bigcup_{k=1}^\infty f^k(U)\right) = 0,
\end{equation}
for every neighbourhood $U$ of $x$. Here, if $S \subset \R^d$, then cap $S = 0$ means that $S$ is, in a precise sense, a ``small'' set; we refer to Section~\ref{Sdefs} for a full definition. It is known \cite[Theorem 1.1]{MR3265283} that if $f$ is a quasiregular map of transcendental type, then card $(J(f)) =\infty$. 

We are not aware that any author has studied the complement of the Julia set of a general quasiregular map. We define the \emph{quasi-Fatou set $QF(f)$} as the complement in $\R^d$ of the Julia set. Thus the quasi-Fatou set is an open set which, if non-empty, has the Julia set as its boundary. Observe that we make no assumptions about the normality of the family of iterates $(f^k)_{k\geq 0}$ in a \emph{quasi-Fatou component}, which is a connected component of the quasi-Fatou set. A set $S$ is said to be \emph{completely invariant} if $x \in S$ implies that $f(x) \in S$ and $f^{-1}(x) \subset S$. It follows from the definitions that both the Julia and quasi-Fatou sets are completely invariant. Note that it follows from \cite[Theorem 1.2]{MR3265283} that if $f$ is a {\tef}, then $QF(f)$ is simply the usual Fatou set.

If $G\subset \R^d$ is a domain, then we denote by $T(G)$ the \emph{topological hull} of the domain, in other words the union of $G$ with its bounded complementary components. Note that unlike, for example, in \cite{MR3215194}, we allow $G$ to be unbounded in this definition. It is well known that if $G\subset\R^d$ is a domain, then $T(G)$ is also a domain; see, for example \cite[Exercise 2.13.6(c)]{MR3014916}. If $G = T(G)$, then we say that $G$ is \emph{full}. Observe that a domain in the complex plane is full if and only if it is simply connected. If $G \ne T(G)$, then we say that $G$ is \emph{hollow}.

Our first result is a generalisation of Baker's result in \cite{MR0153842}.
\begin{theorem}
\label{Texists}
For each $d \geq 2$, there exists a quasiregular map of transcendental type $f : \mathbb{R}^d \to \mathbb{R}^d$ such that $QF(f)$ has a component which is hollow.
\end{theorem}
\begin{remark}\normalfont
If $\R^d=\R^2$, which we can identify with $\C$, we can take $f$ to be a {\tef} with a multiply connected Fatou component. Our construction in dimensions greater than two is an analogue of Baker's original construction, applied using the functions in \cite{MR2053562}.
\end{remark}
\begin{remark}\normalfont
When Baker constructed a multiply connected Fatou component in \cite{MR0153842} he did not know if it was bounded. In fact, it was over ten years before he showed \cite{MR0419759} that this is indeed the case. He later showed \cite{MR759304} that, in fact, this is \emph{always} the case for a {\tef}. We do not know, however, if the hollow quasi-Fatou component of the function in Theorem~\ref{Texists} is bounded or not.
\end{remark}

We prove two results, which concern the cases that a hollow quasi-Fatou component is either bounded or unbounded. The first shows that quasi-Fatou components of quasiregular maps of transcendental type that are both bounded and hollow have properties very similar to those of multiply connected Fatou components of {\tef}s. In order to state this result we need to give a number of definitions, all of which are familiar from complex dynamics.

If $U_0$ is a quasi-Fatou component of $f$, then we let $U_k$ be the quasi-Fatou component of $f$ containing $f^k(U_0)$, for $k\isnatural$. If $U_n \ne U_m$, for $n \ne m$, then we say that $U_0$ is \emph{wandering}.

The \emph{escaping set}
$$
 I(f) = \{x : f^k(x)\rightarrow\infty\text{ as }k\rightarrow\infty\}
$$
was first studied for a general {\tef} in \cite{MR1102727}, and plays an important role in complex dynamics. This set was first studied for a quasiregular map in \cite{MR2448586}.

We also use two subsets of the escaping set which themselves now have an important role in complex dynamics. The \emph{fast escaping set $A(f)$} can be defined by
\begin{equation}
\label{Adef}
A(f) = \{x : \text{there exists } \ell \isnatural \text{ such that } |f^{k+\ell}(x)| \geq M^k(R, f), \text{ for } k \isnatural\}.
\end{equation}
Here $M(R,f)$ denotes the \emph{maximum modulus function} $$M(R,f) = \max_{|x| = R} |f(x)|, \qfor R>0,$$ $M^k(R, f)$ denotes the $k$th iterate of $M(R, f)$ with respect to the first variable, and $R > 0$ can be taken to be any value such that $M^k(R,f)\rightarrow\infty$ as $k\rightarrow\infty$. For a {\tef} this form of the definition of $A(f)$ was first used in \cite{Rippon01102012}. This definition was first used for a quasiregular map of transcendental type in \cite{MR3215194}, where it was shown to be independent of $R$, and equivalent to two other definitions.

We also define the set $$A_R(f) = \{x : |f^k(x)| \geq M^k(R, f), \text{ for } k \isnatural\}.$$

It is known that for a {\tef} $f$, the sets $A_R(f), A(f)$ and $I(f)$ can have a structure known as a {\spw}, which is defined as follows.
\begin{definition}
A set $E\subset\R^d$ is a {\spw} if $E$ is connected, and there exists a sequence $(G_n)_{n\isnatural}$ of bounded full domains with $G_n \subset G_{n+1},$ for $n\isnatural$, $\partial G_n \subset E,$ for $n\isnatural$, and $\cup_{n\isnatural} G_n = \R^d$.
\end{definition}

We are now able to state our result. Here we denote the Euclidean distance from a point $x$ to a set $U \subset \R^d$ by $\operatorname{dist}(x,U) = \inf_{y\in U} |x - y|$. We say that a set $U \subset \R^d$ \emph{surrounds} a set $V \subset \R^d$ if $V$ is contained in a bounded component of the complement of $U$.
\begin{theorem}
\label{Tboundednottopconv}
Suppose that $f : \mathbb{R}^d \to \mathbb{R}^d$ is a quasiregular function of transcendental type. Suppose that $U_0$ is a quasi-Fatou component of $f$ which is bounded and hollow. Let $R > 0$ be such that $M^k(R,f)\rightarrow\infty$ as $k\rightarrow\infty$. Then:
\begin{enumerate}[(a)]
\item each $U_k$ is bounded and hollow, $U_{k+1}$ surrounds $U_k$ for all large $k$, and also $\operatorname{dist}(0,U_k)\rightarrow\infty$ as $k\rightarrow\infty$;\label{Tparta}
\item there exists $\ell\isnatural$ such that $\overline{U}_\ell \subset A_R(f)$, and so $\overline{U}_0 \subset A(f)$;\label{Tpartb}
\item the sets $A_R(f)$, $A(f)$ and $I(f)$ are {\spws}.\label{Tpartc}
\end{enumerate}
\end{theorem}
\begin{remark}\normalfont
If $f$ is a {\tef} and $U_0$ is a Fatou component of $f$, then the assumption that $U_0$ is bounded is not required. In this case part (\ref{Tparta}) is a result of Baker \cite[Theorem 3.1]{MR759304}, and part (\ref{Tpartb}) and part (\ref{Tpartc}) are due to Rippon and Stallard \cite[Theorem 4.4 and Corollary 6.1]{Rippon01102012}. We note that it follows from part (\ref{Tparta}) that $U_0$ is wandering.
\end{remark}
Our second result, which concerns quasi-Fatou components of quasiregular maps of transcendental type which are unbounded and hollow, is as follows.
\begin{theorem}
\label{Tunboundednottopconv}
Suppose that $f : \mathbb{R}^d \to \mathbb{R}^d$ is a quasiregular function of transcendental type. Suppose that $U$ is a quasi-Fatou component of $f$ which is unbounded and hollow. Then $U$ has no unbounded boundary components and is completely invariant, and all other quasi-Fatou components of $f$ are full.
\end{theorem}
\begin{remark}\normalfont
If $f$ is a {\tef} and $U$ is a Fatou component of $f$, then this result is due to T{\"o}pfer \cite{MR0001293}, although Baker \cite[Theorem 3.1]{MR759304} later showed that all multiply connected Fatou components of a {\tef} are, in fact, bounded.
\end{remark}

The following corollary of Theorem~\ref{Tboundednottopconv} and Theorem~\ref{Tunboundednottopconv} is immediate, and is a generalisation of a well-known fact in complex dynamics.
\begin{corollary}
Suppose that $f : \mathbb{R}^d \to \mathbb{R}^d$ is a quasiregular function of transcendental type. If $J(f)$ has a bounded component, then all components of $J(f)$ are bounded.
\end{corollary}

Finally we consider the implications of two possible configurations of the Julia set of a quasiregular map, $f$, of transcendental type. Firstly, it is known \cite[Theorem 1.1]{MR3265357} that
\begin{equation}
\label{JfinpartialAf}
J(f) \subset \partial A(f).
\end{equation}
It is also known \cite[Theorem 1.2]{MR3265357} that if
\begin{equation*}
\liminf_{r\rightarrow\infty}\frac{\log \log M(r,f)}{\log\log r} = \infty,
\end{equation*}
then $J(f) = \partial A(f)$. However, it is not known if it is possible for $QF(f) \cap \partial A(f)$ to be non-empty. Secondly, it is known that for many quasiregular maps of transcendental type, $J(f)$ is perfect. The paper \cite{MR3265283} shows that this holds in a variety of circumstances. It is not known, however, if it is ever possible for $J(f)$ to have an isolated point. 

The following theorem considers the case that either $J(f) \ne \partial A(f)$, or that $J(f)$ has an isolated point. Note that if $f$ is a {\tef}, then $J(f) = \partial A(f)$ \cite[Theorem 5.1]{Rippon01102012} and $J(f)$ is perfect \cite[Theorem 3]{MR1216719}.
\begin{theorem}
\label{Tpathological}
Suppose that $f : \mathbb{R}^d \to \mathbb{R}^d$ is a quasiregular function of transcendental type. Suppose that $QF(f) \cap \partial A(f) \ne\emptyset$, or that $J(f)$ has an isolated point. Then $f$ has a quasi-Fatou component which is unbounded and hollow.
\end{theorem}

A key tool in the proof of these results is the following theorem, which may be of independent interest. Here we denote by $B(a,r)$ the open ball of Euclidean radius~$r$, centred at a point $a \in\R^d$.
\begin{theorem}
\label{TpartialA}
Suppose that $f : \mathbb{R}^d \to \mathbb{R}^d$ is a quasiregular function of transcendental type. Let $r>0$ be sufficiently large that $M^k(r, f)\rightarrow\infty$ as $k\rightarrow\infty$. Suppose that $G$ is a domain, and that there exist $x, y \in G$ and $\ell_0, k_0\isnatural$ such that
\begin{equation}
\label{apart}
|f^{k+\ell_0}(y)| \leq M^{k-1}(r, f) < M^k(r, f) \leq  |f^{k+\ell_0}(x)|, \qfor k \geq k_0.
\end{equation}
Then there exists $\ell\isnatural$ such that
\begin{equation}
\label{TpartialAeq}
B(0, M^k(r,f)) \subset T(f^{k+\ell}(G)), \qfor k\isnatural.
\end{equation}
\end{theorem}
Note that the hypothesis (\ref{apart}) is satisfied if $G \cap \partial A(f) \ne \emptyset$. We give this slightly more general result in view of potential applications.

The proof of Theorem~\ref{TpartialA} uses a certain conformally invariant metric. Very roughly, the theorem holds since $f$ cannot increase this metric too much, whereas the points $x$ and $y$ must iterate far apart in the Euclidean metric. This can be used to show that the boundary of the topological hull of $f^k(G)$ must be far from these iterates, for large values of $k$.\\

The structure of this paper is as follows. First, in Section~\ref{Sdefs} we recall the definitions of quasiregularity, capacity and the modulus of a curve family, and we give some known results required in the rest of the paper. Next, in Section~\ref{Sexists} we prove Theorem~\ref{Texists}. In Section~\ref{SpartialA} we prove Theorem~\ref{TpartialA}. Finally in Section~\ref{Srest} we prove Theorem~\ref{Tboundednottopconv}, Theorem~\ref{Tunboundednottopconv} and Theorem~\ref{Tpathological}.
%
%
%
%
%
%
\section{Quasiregular maps, capacity and the modulus of a curve family}
\label{Sdefs}
We refer to the monographs \cite{MR1238941} and \cite{MR950174} for a detailed treatment of quasiregular maps. Here we only recall the definition and the main properties used.

If $d\geq 2$, $G \subset \R^d$ is a domain and $1 \leq p < \infty$, then the \emph{Sobolev space} $W^1_{p,loc}(G)$ consists of the functions $f : G \to \R^d$ for which all first order weak partial derivatives exist and are locally in $L^p$. If $f \in W^1_{d,loc}(G)$ is continuous, then $f$ is \emph{quasiregular} if there exists a constant $K_O \geq 1$ such that
\begin{equation}
\label{KOeq}
|D f(x)|^d \leq K_O J_f(x) \quad a.e.,
\end{equation}
where $D f(x)$ denotes the derivative,
$$
|D f(x)| = \sup_{|h|=1} |Df(x)(h)|
$$
denotes the norm of the derivative, and $J_f (x)$ denotes the Jacobian determinant. 

We also let
$$
\ell(Df(x)) = \inf_{|h|=1} |Df(x)(h)|.
$$
The condition that (\ref{KOeq}) holds for some $K_O \geq 1$ is equivalent to the condition
that
\begin{equation}
\label{KIeq}
\ell(D f(x))^d \geq K_I J_f(x) \quad a.e.,
\end{equation}
for some $K_I \geq 1$. The smallest constants $K_O$ and $K_I$ such that (\ref{KOeq}) and (\ref{KIeq}) hold are called the \emph{outer and inner dilatation} of $f$ and are denoted by $K_O (f)$ and $K_I (f)$. The \emph{dilatation} of $f$ is denoted by
$$
K(f) = \max\{K_I (f),K_O (f)\}.
$$
We say that $f$ is $K$-quasiregular if $K(f) \leq K$, for some $K \geq 1$.

If $f$ and $g$ are quasiregular, with $f$ defined in the range of $g$, then $f \circ g$ is quasiregular and \cite[Theorem II.6.8]{MR1238941}
\begin{equation}
\label{Keq}
K_I(f \circ g) \leq K_I(f)K_I(g).
\end{equation}

Many properties of holomorphic functions extend to quasiregular maps. In particular, we frequently use the fact that non-constant quasiregular maps are open and discrete.

We need a result on the growth of the maximum modulus of a quasiregular map of transcendental type; see \cite[Lemma 3.4]{MR2248829}, \cite[Corollary 4.3]{MR1799670}.
\begin{lemma}
\label{Blemm}
Suppose that $f : \mathbb{R}^d \to \mathbb{R}^d$ is a quasiregular function of transcendental type. Then
$$
\lim_{r\rightarrow\infty} \frac{\log M(r, f)}{\log r} = \infty.
$$
\end{lemma}
Lemma~\ref{Blemm} has the following corollary, which is immediate.
\begin{corollary}
\label{Bcorr}
Suppose that $f : \mathbb{R}^d \to \mathbb{R}^d$ is a quasiregular function of transcendental type. If $r>0$ is sufficiently large that $M^k(r, f)\rightarrow\infty$ as $k\rightarrow\infty$, then
\begin{equation*}
\lim_{k\rightarrow\infty} \frac{\log \log M^k(r, f)}{k} = \infty.
\end{equation*}
\end{corollary}

An important tool in the study of quasiregular mappings is the capacity of a condenser, and we recall this concept briefly. Suppose that $A \subset \R^d$ is open, and $C \subset A$ is non-empty and compact. The pair $(A,C)$ is called a \emph{condenser}, and its \emph{capacity}, which we denote by cap$(A,C)$, is defined by
$$
\operatorname{cap}(A,C) = \inf_u \int_A |\nabla u|^d \ dm,
$$
where the infimum is taken over all non-negative functions $u \in C^\infty_0(A)$ which satisfy $u(x) \geq 1$, for $x \in C$.

If $C$ is compact and cap$(A,C) = 0$ for some bounded open set $A$ containing $C$, then cap$(A ' ,C) = 0$ for every bounded open set $A'$ containing $C$; see \cite[Lemma III.2.2]{MR1238941}. In this case we write cap $C = 0$. Otherwise we write cap $C > 0$. For an unbounded closed set $C \subset \R^d$, we write cap $C = 0$ if cap $C' = 0$ for every compact subset $C'$ of $C$. 

If $C$ is closed and cap $C = 0$, then $C$ is, in a sense, a small set. In particular we use the fact \cite[VII.1.15]{MR1238941} that the Hausdorff dimension of $C$ is zero. In particular, if $C$ has an interior point, then cap $C>0$. \\

A second important tool in the study of quasiregular maps is the concept of the modulus of a curve family; we refer to \cite[Chapter II]{MR1238941} and \cite[Chapter 2]{MR950174} for a detailed discussion. Suppose that $\Gamma$ is a family of paths in $\R^d$. A non-negative Borel function $\rho : \R^d \to \R \cup \{\infty\}$ is called \emph{admissible} if $\int_\gamma \rho \ ds \geq 1$, for all locally rectifiable paths $\gamma \in \Gamma$. We let $\mathcal{F}(\Gamma)$ be the family of all admissible Borel functions, and let the \emph{modulus} of $\Gamma$ be given by
$$
M(\Gamma) = \inf_{\rho \in \mathcal{F}(\Gamma)} \int_{\R^d}\rho^d \ dm.
$$

Suppose that $G \subset \R^d$ is a domain, and $E, F$ are subsets of $\overline{G}$. We denote by $\Delta(E,F;G)$ the family of all paths which have one endpoint in $E$, one endpoint in $F$, and which otherwise are in $G$. 

A connection between the capacity of a condenser and the modulus of a path family is the fact \cite[Proposition II.10.2]{MR1238941} that if $E$ is a compact subset of $G$, then 
$$
\text{cap}(G,E) = M(\Delta(E,\partial G;G)).
$$

Next we introduce a conformal invariant which is a useful alternative to the hyperbolic metric when working with quasiregular maps. Let $G \subset \R^d$ be a domain, and define a function $\mu_G$ by
$$
  \mu_G(x,y) = \inf_{C_{xy}} M(\Delta(C_{xy}, \partial G; G)),\qfor x, y \in G,
$$
where the infimum is taken over curves $C_{xy}$ which are contained in $G$ and join $x$ and $y$; see \cite[p.103]{MR950174} for this definition and more background. It is known that $\mu_G(x,y)$ is a conformal invariant, and is a metric if cap $\partial G > 0$. It is noted in \cite{MR950174} that if $D \subset G$ is a domain, then
\begin{equation}
\label{smallerdomain}
\mu_D(x,y) \geq \mu_G(x,y), \qfor  x, y \in D.
\end{equation}

The ``transformation formula'' for $\mu_G$ is as follows \cite[Theorem 10.18]{MR950174}. 
\begin{lemma}
\label{l3}
Suppose that $f : G \to \mathbb{R}^d$ is a non-constant quasiregular mapping. Then
$$
\mu_{f(G)}(f(a), f(b)) \leq K_I(f) \mu_G(a, b), \qfor a, b \in G.
$$
\end{lemma}
We require the following estimate for $\mu_G$ \cite[Theorem 8.31]{MR950174}. (Here we write $\partial_\infty G$ for the boundary of $G$ taken in $\overline{\R^d}$.)
\begin{lemma}
\label{muestimate}
Suppose that $G$ is a proper subdomain of $\mathbb{R}^d$ such that $\partial_\infty G$ is connected. Then there exists a constant $c_d$, depending only on $d$, such that
$$
\mu_G(a, b) \geq c_d \log\left(1 + \frac{|a-b|}{\min\{\operatorname{dist}(a, \partial G), \operatorname{dist}(b, \partial G)\}}\right), \qfor a, b \in G.
$$
\end{lemma}
%
%
%
%
%
%
\section{Proof of Theorem~\ref{Texists}}
\label{Sexists}
We construct our function using a theorem of Drasin and Sastry in \cite{MR2053562}. We need to introduce some terminology before we can state the result we use from that paper.

Suppose that $\nu : \mathbb{R} \to \mathbb{R}$ is continuous, positive and increasing, and that $\nu(r)~\rightarrow~\infty$ as $r\rightarrow\infty$. Let $n_0$ be an integer greater than $\nu(0)$, and define a sequence of integers $(r_n)_{n\geq n_0}$ by
\begin{equation}
\label{rneq}
r_n = \max \{ r : \nu(r) = n\}.
\end{equation}

In \cite{MR2053562} it is assumed that
\begin{equation}
\label{constraint1}
r\nu'(r) < \nu(r)/2 \quad\text{and}\quad r \nu'(r) = o(\nu(r)) \text{ as } r\rightarrow\infty.
\end{equation}
In fact (\ref{constraint1}) is only used in \cite{MR2053562} to deduce that
\begin{equation}
\label{constraint2}
n\log\frac{r_{n+1}}{r_n} \rightarrow\infty \text{ as } n\rightarrow\infty.
\end{equation}
In our case it is somewhat easier to check (\ref{constraint2}) directly, and so we use this equation in the statement of Drasin and Sastry's theorem.

It is helpful to write down a formula for an exceptional set. For $\epsilon \in (0,1)$, we define a union of closed intervals
$$
E_\epsilon = \bigcup_{n > n_0} \left[\epsilon r_{n}, r_n\right].
$$

Define a function
\begin{equation}
\label{Ldef}
L(r) = \exp \left(\int_1^r \frac{\nu(t)}{t} dt\right).
\end{equation}

Finally, it is helpful to work with the \emph{maximum norm}, which is defined by
$$
\maxnorm{x} = \max\{|x_1|, \ldots, |x_d|\}, \qfor x = (x_1, \ldots, x_d) \in \R^d.
$$

A statement of Drasin and Sastry's result, which includes part of their construction, is as follows.
\begin{theorem}
\label{TDS}
Suppose that $d \geq 2$, that $\nu(r)$ and $L(r)$ are as above, and that (\ref{constraint2}) is satisfied. Then there exist a quasiregular map of transcendental type $f : \mathbb{R}^d \to \mathbb{R}^d$, and constants $c, C, R_0 > 0$ and $\epsilon \in (0,1)$ such that
\begin{equation}
\label{growth}
c L(\maxnorm{x}) < \maxnorm{f(x)} < CL(\maxnorm{x}), \qfor \maxnorm{x} \in [R_0, \infty)\backslash E_\epsilon.
\end{equation}
\end{theorem}
The form of the exceptional set $E_\epsilon$ in the statement of Theorem~\ref{TDS}, which is not made explicit in \cite{MR2053562}, can be obtained as follows. First, we can deduce from \cite[Equation (2.10)]{MR2053562}, together with the paragraph preceding it, that there exists $\epsilon \in (0,1)$ such that if $\epsilon r_{n+1} > r_n$, then the interval $[r_n, \epsilon r_{n+1}]$ lies in the interval labelled $J^0_n$ by Drasin and Sastry, for all sufficiently large values of $n$. Second, the fact that (\ref{growth}) holds when $\maxnorm{x} \in J^0_n$ follows from \cite[Equation (3.6)]{MR2053562}.
\begin{proof}[Proof of Theorem~\ref{Texists}]
Roughly speaking, we construct a quasiregular function $f$ of transcendental type with the following properties:
\begin{enumerate}[(a)]
\item the function $f$ behaves like a power map in very large ``square rings'';\label{exp1}
\item subrings of these can be defined in such a way that $f$ maps these subrings into each other.\label{exp2}
\end{enumerate}

Note that this is, essentially, the same idea as Baker's original construction of a multiply connected Fatou component of a {\tef} using an infinite product. Properties (\ref{exp1}) and (\ref{exp2}) are achieved iteratively: if $r_n$ is defined, then $\nu$ and $r_{n+1}$ are chosen so that $r_{n+1}$ is much large than $r_n$ (to satisfy (\ref{exp1})), and so that if $\maxnorm{x} = r_n$, then $\maxnorm{f(x)}$ is approximately equal to $r_{n+1}$ (which leads to (\ref{exp2})). Property (\ref{exp2}) then ensures that these subrings lie in $QF(f)$, and the theorem follows.

We now give the full detail of the proof. Let $d\geq2$ be an integer. We first construct a function $\nu$ with certain properties. We then invoke Theorem~\ref{TDS} to obtain a quasiregular map $f : \mathbb{R}^d \to \mathbb{R}^d$ such that (\ref{growth}) is satisfied. Finally we prove that this function has a hollow quasi-Fatou component.

The construction of $\nu$ is as follows. First choose an integer $n_0 \geq 2$ and a real number $R'>4$. Set $\nu(r) = n_0$, for $0 \leq r \leq R'$. Since we shall ensure that $\nu(r) > n_0$, for $r > R'$, this implies that $r_{n_0} = R'$.

We complete the definition of $\nu$, and the sequence $(r_n)_{n > n_0}$ iteratively. Suppose that $n\geq n_0$, that $r_n$ has been defined, and that $\nu(r)$ has been defined for $0 \leq r \leq r_n$. We set
\begin{equation}
\label{rndef}
r_{n+1} = L(r_n),
\end{equation}
and let $\nu(r)$ be linear; in other words, we set
\begin{equation}
\label{nudef}
\nu(r) = \frac{r - r_n}{r_{n+1} - r_n} + n, \qfor r_n \leq r \leq r_{n+1}.
\end{equation}

Note that if $n\geq n_0$, then
\begin{equation}
\label{rnbound}
r_{n+1} = L(r_n) = \exp \left(\int_1^{r_n} \frac{\nu(t)}{t} dt\right) \geq \exp \left(\int_1^{r_n} \frac{n_0}{t} dt\right) \geq r_n^{n_0}.
\end{equation}

It follows from the choices of $n_0$ and $R'$ that $\nu$ is positive, continuous and increasing, tends to infinity, and that (\ref{constraint2}) holds. Hence we can let $f : \mathbb{R}^d \to \mathbb{R}^d$ and constants $c, C, R_0 > 0$ and $\epsilon \in (0,1)$ be as in Theorem~\ref{TDS}. \\

We complete the proof of the theorem by showing that $f$ has a hollow quasi-Fatou component. For each sufficiently large integer $n\geq n_0$, define a ``square ring''
$$
A_n = \left\{ x : 2r_n < \maxnorm{x} < \epsilon r_{n+1}\right\}.
$$

We claim that for all sufficiently large values of $n$, we have that $f(A_n) \subset A_{n+1}$. It then follows from (\ref{Juliadef}) that $A_n \subset QF(f)$. It is easy to see that the theorem follows from this fact, since the Julia set of $f$ is not empty.

Choose $N_0 \geq n_0$ sufficiently large that, for $n\geq N_0$ we have that $2r_n < \epsilon r_{n+1}$, that $c 2^n \geq 2$, and that $C \epsilon^n \leq\epsilon$.

Suppose that $n\geq N_0$, and that $x \in A_n$. Then, by (\ref{growth}), (\ref{rndef}), (\ref{nudef}) and since $L(r)$ is increasing
\begin{align*}
\maxnorm{f(x)} 
       &\geq c L(2r_n) \\
       &=    c \exp \left(\int_1^{r_n} \frac{\nu(t)}{t} dt \right) \  \exp \left(\int_{r_n}^{2r_n} \frac{\nu(t)}{t} dt\right) \\
       &\geq c L(r_n) \exp \left(\int_{r_n}^{2r_n} \frac{n}{t} dt \right)\\
       &=    c 2^n L(r_n) \\
       &\geq 2 L(r_n) = 2r_{n+1}.
\end{align*}
For the same reasons,
\begin{align*}
\maxnorm{f(x)} 
       &\leq C L\left(\epsilon r_{n+1}\right) \\
       &=    C \exp \left(\int_1^{r_{n+1}} \frac{\nu(t)}{t} dt \right) \exp \left(-\int_{\epsilon r_{n+1}}^{r_{n+1}} \frac{\nu(t)}{t} dt\right) \\
       &\leq C L(r_{n+1})\exp \left(-\int_{\epsilon r_{n+1}}^{r_{n+1}} \frac{n}{t} dt\right) \\
       &=    \epsilon^{n} C L(r_{n+1})
       \leq \epsilon L(r_{n+1}) = \epsilon r_{n+2}.
\end{align*}
Hence $f(x) \in A_{n+1}$, and this completes the proof.
\end{proof}
\begin{remark}\normalfont
For the function $f$ constructed above, we do not know if $J(f) = \partial A(f)$. The growth of $f$ is too slow to use the result \cite[Theorem 1.2]{MR3265357} mentioned in the introduction. In fact, it follows from (\ref{rnbound}) that
$$
\log r_{n+n_0} \geq n_0^n \log r_{n_0}, \qfor n\isnatural.
$$
We deduce, by (\ref{Ldef}) and (\ref{growth}), that
\begin{align*}
\liminf_{r\rightarrow\infty} \frac{\log\log M(r, f)}{\log \log r}
  &\leq \liminf_{n\rightarrow\infty} \frac{\log \log M(2r_n,f)}{\log \log 2r_n} \\
  &\leq \liminf_{n\rightarrow\infty} \frac{\log \log (C \sqrt{d} L(2 r_n))}{\log \log 2r_n} \\
  &\leq \liminf_{n\rightarrow\infty} \frac{\log \log (C \sqrt{d}(2 r_n)^{n+1})}{\log \log 2r_n} \\
  &\leq \liminf_{n\rightarrow\infty} \frac{\log 2(n+1) + \log \log 2 r_n}{\log \log 2r_n} = 1.
\end{align*}
\end{remark}
%
%
%
%
%
%
\section{Proof of Theorem~\ref{TpartialA}}
\label{SpartialA}
In general, if $H$ is a domain, then we let the \emph{outer boundary} of $H$ be denoted by $\partial_{out} H = \partial T(H)$.

Suppose that $f : \mathbb{R}^d \to \mathbb{R}^d$ is a quasiregular function of transcendental type, and that $G$ is a domain. We can assume we have $T(f^k(G)) \ne \R^d$, for $k\isnatural$, as otherwise there is nothing to prove.

Let $x, y \in G$, $r>0$ and $\ell_0, k_0 \isnatural$ be as in the statement of the theorem. By Lemma~\ref{Blemm} and (\ref{apart}), we can assume that
\begin{equation}
\label{yissmall}
2|f^{k}(y)| \leq |f^{k}(x)|, \qfor \text{all sufficiently large } k.
\end{equation}
We claim that there exists $k_2\isnatural$ such that
\begin{equation}
\label{T4toprove}
\operatorname{dist }(0, \partial_{out} f^{k+\ell_0}(G)) \geq M^{k-1}(r,f)), \qfor k \geq k_2.
\end{equation}
Since $|f^{k+\ell_0}(y)| \leq  M^{k-1}(r,f)$, for all $k \geq k_0$, it is easy to see that this claim implies that
$$
B(0, M^k(r, f)) \subset B(0, M^{k+k_2 - 1}(r, f)) \subset T(f^{k+k_2+\ell_0}(G)),\qfor k\isnatural,
$$
from which Theorem~\ref{TpartialA} follows with $\ell = k_2 + \ell_0$.\\

To prove (\ref{T4toprove}), we first deduce from Lemma~\ref{l3}, Lemma~\ref{muestimate}, (\ref{Keq}), (\ref{smallerdomain}) and (\ref{yissmall}) that, for all large $k$,
\begin{align*}
K_I(f)^k \mu_G(x,y) &\geq K_I(f^k) \mu_G(x,y) \\
                    &\geq\mu_{f^k(G)}(f^k(x), f^k(y)) \\
                    &\geq\mu_{T(f^k(G))}(f^k(x), f^k(y)) \\
                    &\geq c_d \log\left(1 + \frac{|f^k(x)-f^k(y)|}{\operatorname{dist}(f^k(y), \partial_{out} f^k(G))}\right) \\
                    &\geq c_d (\log(|f^k(x)|) - \log(\operatorname{dist}(f^k(y), \partial_{out} f^k(G))) - \log 2).
\end{align*}
We deduce by (\ref{apart}) that there exist $c>K_I(f)$ and $k_1 \geq k_0$ such that
\begin{align}
\log M^{k}(r, f) &\leq \log |f^{k+\ell_0}(x)| \nonumber \\
                 &\leq c^k + \log (\operatorname{dist}(f^{k+\ell_0}(y), \partial_{out} f^{k+\ell_0}(G))) \nonumber \\
                 &\leq c^k + \log (\operatorname{dist}(0, \partial_{out} f^{k+\ell_0}(G)) + |f^{k+\ell_0}(y)|), \nonumber\\
                 &\leq c^k + \log (\operatorname{dist}(0, \partial_{out} f^{k+\ell_0}(G)) + M^{k-1}(r,f)), \qfor k\geq k_1.\label{e9}
\end{align}

By Lemma~\ref{Blemm} and Corollary~\ref{Bcorr}, we can choose $k_2 \geq k_1$ sufficiently large that all the following hold for $k\geq k_2$:
\begin{enumerate}[(i)]
\item $\log \log M^{k-1}(r,f) \geq 2k \log c$;
\item $\log M^k(r,f) \geq 2 \log M^{k-1}(r,f)$;
\item $c^{2k} \geq c^k + \log 2$.
\end{enumerate}
We deduce by (\ref{e9}) that, for $k\geq k_2$,
\begin{align*}
c^k + \log (\operatorname{dist}(0, \partial_{out} f^{k+\ell_0}(G)) + M^{k-1}(r,f))
  &\geq \log M^{k}(r, f) \\
  &\geq 2 \log M^{k-1}(r,f) \\
  &\geq \log M^{k-1}(r,f) + c^{2k} \\
  &\geq \log M^{k-1}(r,f) + c^k + \log 2.
\end{align*}
Equation (\ref{T4toprove}) follows, and this completes the proof.
%
%
%
%
%
%
%
\section{Proofs of Theorem~\ref{Tboundednottopconv}, Theorem~\ref{Tunboundednottopconv} and Theorem~\ref{Tpathological}}
\label{Srest}
We require the following topological lemma.
\begin{lemma}
\label{Ltopconv}
Suppose that $f : \mathbb{R}^d \to \mathbb{R}^d$ is a quasiregular function, and that $U$ is a quasi-Fatou component of $f$ which is bounded and full. Then $f(U)$ is bounded and full, and the quasi-Fatou component of $f$ containing $f(U)$ is equal to $f(U)$.
\end{lemma}
\begin{proof}
Suppose that $U$ is a quasi-Fatou component of $f$ which is bounded and full. Clearly $f(U)$ is bounded. We claim that $\partial f(U)$ is connected. For, since $f$ is an open map and $U$ is bounded, $\partial f(U) \subset f(\partial U)$. Moreover, since $f$ is continuous we have
\begin{equation}
\label{partialUeq}
f(\partial U) \subset f(\overline{U}) \subset \overline{f(U)}.
\end{equation}
Since $f(\partial U) \subset J(f)$ and $f(U) \subset QF(f)$, it follows from (\ref{partialUeq}) that
$$
f(\partial U) \subset \overline{f(U)}\backslash f(U) = \partial f(U).
$$

We deduce that $f(\partial U) = \partial f(U)$. Since $U$ is bounded and full, $\partial U$ is connected. Hence $\partial f(U)$ is connected, as claimed.

Since $f(U)$ is bounded, if it were also hollow, then $\partial f(U)$ would have at least two components. Hence $f(U)$ is full.

Since $\partial f(U) = f(\partial U) \subset J(f)$, we also have that the quasi-Fatou component of $f$ containing $f(U)$ is equal to $f(U)$.
\end{proof}
We now prove Theorem~\ref{Tboundednottopconv}.
\begin{proof}[Proof of Theorem~\ref{Tboundednottopconv}]
Suppose that $f : \mathbb{R}^d \to \mathbb{R}^d$ is a quasiregular function of transcendental type, and that $U_0$ is a quasi-Fatou component of $f$ which is bounded and hollow.

Since $T(U_0) \cap J(f) \ne \emptyset$, it follows from (\ref{JfinpartialAf}) that $T(U_0) \cap A(f) \ne \emptyset$. Since $U_0$ is bounded, it follows by the maximum principle that $U_0$ is wandering.

Our proof now splits into two parts. First we show that there exists a bounded complementary component of $U_0$ which contains a point of $J(f)$ in its interior. We then use this fact to prove the theorem.

Since $T(U_0) \cap J(f) \ne \emptyset$, it follows from (\ref{Juliadef}) that there exist $x \in T(U_0)$ and $k \geq 1$ such that $f^k(x) \in U_0$. Let $V$ be the quasi-Fatou component of $f$ containing $x$. Then $V \subset T(U_0)$ and, since $U_0$ is wandering, we have $V \cap U_0 = \emptyset$.

Moreover, $V$ is hollow by Lemma~\ref{Ltopconv}. Hence $T(V)$ contains a point of $J(f)$. Let $E$ be the component of the complement of $U_0$ containing $T(V)$. It follows that $\operatorname{int } E \cap J(f) \ne \emptyset$, as claimed.

Let $G_0 = \operatorname{int } E$. It follows by the maximum principle that $f(G_0) \subset T(f(U_0))$. Moreover, $\partial G_0 \subset J(f)$, and so $f(\partial G_0)$ cannot meet $f(U_0)$. It follows that we can let $G_1$ be the interior of the bounded complementary component of $f(U_0)$ which contains $f(G_0)$. In general, we can let $G_k$ be the interior of the bounded complementary component of $f^k(U_0)$ which contains $f^k(G_0)$, for $k\isnatural$. Clearly $G_k$ is full, for $k\isnatural$.

Suppose that $k \isnatural$. Since $\partial G_0 \subset J(f)$, we have $f^k(U_0) \cap f^k(\partial G_0) = \emptyset$, and so $f^k(\partial G_0) \subset \overline{G_k}$. Clearly $f^k(\partial G_0) \cap G_k = \emptyset$, since otherwise $G_k$ contains a point of $f^k(U_0)$. Hence $f^k(\partial G_0) \subset \partial G_k$. Thus, since $f^k$ is an open map and $G_0$ is bounded, $$\partial f^k(G_0) \subset f^k(\partial G_0) \subset \partial G_k.$$ Since $G_k$ is full, we deduce that
\begin{equation}
\label{nicesets}
f^k(G_0) = G_k, \qfor k \isnatural.
\end{equation}

Let $R>0$ be sufficiently large that $M^k(R, f)\rightarrow\infty$ as $k\rightarrow\infty$. Since we have that $G_0 \cap J(f) \ne \emptyset$, it follows by (\ref{JfinpartialAf}) that $G_0 \cap \partial A(f) \ne \emptyset$, and so we can apply Theorem~\ref{TpartialA} with $G = G_0$. We deduce, by Theorem~\ref{TpartialA} and (\ref{nicesets}) that there exists $\ell\isnatural$ such that
$$
B(0, M^k(R,f)) \subset T(f^{k+\ell}(G_0)) = T(G_{k+\ell}) = G_{k+\ell}, \qfor k\isnatural.
$$

Theorem~\ref{Tboundednottopconv} part~(\ref{Tparta}) and part (\ref{Tpartb}) follow.\\

To prove part~(\ref{Tpartc}), we note first that it follows from Theorem~\ref{Tboundednottopconv} part (\ref{Tpartb}) and (\ref{JfinpartialAf}) that the complement of $A_R(f)$ has a bounded component. The fact that the sets $A_R(f)$ and $A(f)$ are both {\spws} follows immediately from \cite[Proposition 6.2] {MR3215194}. 

It remains to show that $I(f)$ is a {\spw}. Since $A(f)$ is a {\spw}, and so connected, we can let $I_0$ be the component of $I(f)$ which contains $A(f)$. By definition $I_0$ is a {\spw}. We show that, in fact, $I(f) = I_0$.

First, suppose that $x \in I(f) \cap J(f)$. Then $x \in \partial A(f)$, by (\ref{JfinpartialAf}), and so $x \in \overline{I_0}$. Thus $I_0 \cup \{x\}$ is connected and, since $x \in I(f)$ and $I_0$ is a component of $I(f)$, it follows that $x \in I_0$. Thus
\begin{equation}
\label{IJinI0}
(I(f) \cap J(f)) \subset I_0.
\end{equation}

Next, suppose that $x \in I(f) \cap QF(f)$. Let $V_0$ be the component of $QF(f)$ containing $x$. We claim that $\overline{V_0} \subset  I(f)$. It is easy to see that the fact that $I(f) = I_0$ follows from this claim, by (\ref{IJinI0}).

Recall that $V_n$ (resp. $U_n$) denotes the Fatou component containing $f^n(V_0)$ (resp. $f^n(U_0)$). If $V_n = U_m$, for some $n, m \isnatural$, then $\overline{V_n} \subset A(f)$, by Theorem~\ref{Tboundednottopconv} part (\ref{Tpartb}). So we can assume that 
\begin{equation}
\label{nevermeets}
V_n\ne U_m, \qfor n, m\isnatural.
\end{equation}

For all sufficiently large values of $n$, let $B_n$ denote the intersection of the unbounded complementary component of $U_n$ with $T(U_{n+1})\backslash U_{n+1}$. Since $f^n(x) \rightarrow\infty$, it follows by (\ref{nevermeets}) that there is a sequence $(k_n)_{n\isnatural}$ such that $f^n(x) \in B_{k_n}$, for all sufficiently large $n\isnatural$, and $k_n \rightarrow\infty$ as $n \rightarrow\infty$.

It follows by (\ref{nevermeets}) that $\overline{V_n} \subset \overline{B_{k_n}}$, for all sufficiently large $n\isnatural$, and the result follows.
\end{proof}
Next we prove Theorem~\ref{Tunboundednottopconv}.
\begin{proof}[Proof of Theorem~\ref{Tunboundednottopconv}]
Suppose that $f : \mathbb{R}^d \to \mathbb{R}^d$ is a quasiregular function of transcendental type. Suppose that $U$ is a quasi-Fatou component of $f$ which is unbounded and hollow. 

Let $E$ be a bounded complementary component of $U$. We claim that there is a bounded full domain, $G$, such that $E \subset G$ and $\partial G \subset U$. We prove this claim as follows. For each $n\isnatural$, let $V_n$ be the component of $\{y : \operatorname{dist}(y, U^c) \leq 1/n\}$ that contains $E$.

Suppose that all the sets $V_n$ are unbounded. Let $\rho > 0$ be sufficiently large that $E \subset B(0, \rho/2)$. For each $n\isnatural$, let $V^*_n$ be the component of $V_n \cap \overline{B(0, \rho)}$ containing $E$. It is easy to see that each $V^*_n$ is a continuum. Let $V^* = \bigcap_{n\isnatural} V^*_n$. Then $V^*$ is a nested intersection of continua, and so \cite[Theorem 1.8]{MR1192552} is a continuum. In particular $V^*$ is connected. 

Since $U$ is open, $V^* \subset U^c$. Now $E = V^*$, since $E \subset V^*$ and $E$ is a component of $U^c$. Also, $V^*$ contains the non-empty set $\bigcap_{n\isnatural} (V^*_n \cap \partial B(0, \rho))$, and so also contains a point of modulus $\rho$. This is a contradiction.

It follows that there exists $n\isnatural$ such that $V_n$ is bounded. We let $G$ be the topological hull of the component of $\{y : \operatorname{dist}(y, U^c) < 1/n\}$ that contains $E$. This completes the proof of our claim.

It follows from \cite[Proposition 2.4]{MR3215194} that
\begin{equation}
\label{Ueq}
\partial T(f^k(G)) \subset f^k(\partial T(G)) = f^k(\partial G) \subset f^k(U).
\end{equation}

Since $E \cap J(f) \ne \emptyset$, we have by (\ref{JfinpartialAf}) that $G \cap \partial A(f) \ne \emptyset$. We apply Theorem~\ref{TpartialA} with $r>0$ sufficiently large that $M^k(r, f)\rightarrow\infty$ as $k\rightarrow\infty$.

Since $U$ is unbounded, we deduce from (\ref{TpartialAeq}) that $\partial T(f^k(G)) \cap U\ne \emptyset$ for all sufficiently large values of $k$. Hence, by (\ref{Ueq}), $U$ meets $f^k(U)$ and $f^{k+1}(U)$, for some $k$, and so $U$ is the quasi-Fatou component containing $f(U)$. We deduce similarly by (\ref{TpartialAeq}) and (\ref{Ueq}) that $U$ has no unbounded boundary components.

Next, suppose that $V \ne U$ is a quasi-Fatou component of $f$ such that $f(V) \subset U$. Since $U$ has no unbounded boundary components, all complementary components of $U$ are bounded, and so $V$ is bounded. It follows by Lemma~\ref{Ltopconv} that $V$ is hollow. We deduce a contradiction by Theorem~\ref{Tboundednottopconv} part~(\ref{Tparta}). Hence $U$ is completely invariant.

Finally, it is straightforward to see that Theorem~\ref{Tboundednottopconv} part~(\ref{Tparta}) implies that $f$ has no other hollow quasi-Fatou components.
\end{proof}

We note the following corollary, which is a generalisation of a well-known result for {\tef}s.
\begin{corollary}
Suppose that $f : \mathbb{R}^d \to \mathbb{R}^d$ is a quasiregular function of transcendental type. Suppose that $U$ is a quasi-Fatou component of $f$, and $V$ is the quasi-Fatou component of $f$ containing $f(U)$. Then $U$ is full if and only if $V$ is full.
\end{corollary}
\begin{proof}
If $U$ is hollow, then $V$ is hollow by Theorem~\ref{Tboundednottopconv} and Theorem~\ref{Tunboundednottopconv}. If $U$ is bounded and full, then $V$ is full, by Lemma~\ref{Ltopconv}. Finally, it is easy to see that if $U$ is unbounded and full, then $f$ has no hollow quasi-Fatou components.
\end{proof}

Finally, we prove Theorem~\ref{Tpathological}.
\begin{proof}[Proof of Theorem~\ref{Tpathological}]
Suppose that $f : \mathbb{R}^d \to \mathbb{R}^d$ is a quasiregular function of transcendental type. Suppose first that there exists $x \in QF(f) \cap \partial A(f)$. Let $G$ be a domain containing $x$ and contained in $QF(f)$. We apply Theorem~\ref{TpartialA} with $r>0$ sufficiently large that $M^k(r, f)\rightarrow\infty$ as $k\rightarrow\infty$.

It follows from (\ref{TpartialAeq}), and from the fact that $J(f) \ne \emptyset$, that there exists $k\isnatural$ and a hollow quasi-Fatou component $U$ such that $f^k(G) \subset U$. Since $G \cap A(f)^c \ne \emptyset$, it follows from Theorem~\ref{Tboundednottopconv} part (\ref{Tpartb}) that $U$ is unbounded. This completes the proof of the first part of the theorem.

Suppose second that $J(f)$ has an isolated point, $x$. Let $G$ be a domain such that $G \cap J(f) = \{ x \}$. By (\ref{JfinpartialAf}) we have that $G \cap A(f)^c \ne \emptyset$. Clearly $f$ has a hollow quasi-Fatou component $U$ which contains $G \backslash \{x\}$. Then Theorem~\ref{Tboundednottopconv} part~(\ref{Tpartb}) implies that $U$ is unbounded.
\end{proof}

%
%
%
%
%
%
%
%
%
%
%
\bibliographystyle{acm}
\bibliography{../../Research.References}
\end{document}